\documentclass[11pt,reqno]{amsart}
\usepackage{amscd,amsmath,amssymb,amsthm,amsfonts}
\usepackage{lipsum}
\usepackage{xpatch}
\xpatchcmd{\paragraph}{\normalfont}{{\normalfont\bfseries}}{}{}
\usepackage{cancel}
\usepackage{color,xcolor}
\usepackage{graphics}
\usepackage{inputenc}
\usepackage{fontenc}

\usepackage[english]{babel}

\newcommand{\myspace}{\qquad\qquad\qquad}

\newcommand{\cD}{{\mathcal D}}
 
\newcommand{\cG}{{\mathcal G}}

\newcommand{\cL}{{\mathcal L}}

\newtheorem{theorem}{Theorem}[section]
\newtheorem{lemma}[theorem]{Lemma}
\newtheorem{proposition}[theorem]{Proposition}
\newtheorem{remark}[theorem]{Remark}

\newtheorem{problem}[theorem]{Problem}

\numberwithin{equation}{section}

\usepackage{mathrsfs}
\usepackage[margin=37mm]{geometry}
\newcommand{\R}{\mathbb{R}}
\newcommand{\C}{\mathbb{C}}

\def\ve{\varepsilon}

\date{}

\begin{document}


\title{An optimal control problem for Maxwell's equations}

\author{Francesca Bucci}
\address{Francesca Bucci, Universit\`a degli Studi di Firenze,
Dipartimento di Matematica e Informatica,
Viale Morgagni 65, 50139 Firenze, ITALY
}
\email{francesca.bucci@unifi.it}

\author{Matthias Eller}
\address{Matthias Eller, Georgetown University, 
Department of Mathematics and Statistics,
Georgetown~360, 37th and O Streets NW, Washington DC 20057, USA
}
\email{Matthias.Eller@georgetown.edu}



\begin{abstract}
This article is concerned with the optimal boundary control of the Maxwell system.~We consider a Bolza problem, where the quadratic functional to be minimized penalizes the electromagnetic field at a given final time.
Since the state is weighted in the energy space topology
-- a physically realistic choice --, the property that the optimal cost operator does satisfy 
 the Riccati equation (RE) corresponding to the optimization problem is missed, just like in the case of other significant hyperbolic partial differential equations; however, we prove that this Riccati operator as well as the optimal solution can be recovered by means of approximating problems for which the optimal synthesis holds via proper differential Riccati equations.
In the case of zero conductivity, an explicit representation of the optimal pair is valid
which does not demand the well-posedness of the RE, instead. 
\end{abstract}

\maketitle


\section{Introduction and main result}
In this note we consider the evolution of the electromagnetic field  $y=(e,h)$ in a bounded, open, and
connected set $\Omega$ with a $C^{1,1}$ boundary $\Gamma=\partial\Omega$. 
This evolution is governed by Maxwell's equations
\begin{equation}\label{e:maxwell}
\begin{split}
\ve e_t - \nabla \times h +\sigma e & = 0 
\\ 
\mu h_t + \nabla\times e &= 0 
\end{split}
\qquad \mbox{ in } Q = (0,T) \times \Omega\,,
\end{equation}
and augmented by initial conditions 
\begin{equation} \label{e:IC}
(e,h)(0) = (e_0,h_0)=y_0 \quad \mbox{ in } \Omega\;,
\end{equation}
and the boundary condition
\begin{equation}\label{e:BC}
\nu\times h = g  \quad \mbox{ in } \Sigma :=(0,T) \times \Gamma\;.
\end{equation}
Throughout the electric permittivity $\ve$ and the magnetic permeability $\mu$ are $3\times 3$ real-symmetric, uniformly positive definite matrix functions with elements in the Sobolev space $W^{1,\infty}(\Omega)$ and the conductivity $\sigma$ is a real-symmetric, non-negative definite matrix function with entries in $L^\infty(\Omega)$.
Note that all coefficients are time-independent. 
The exterior unit normal vector along $\Gamma$ is denoted by $\nu$. Note that the vector $\nu \times h$ is a tangential vector field, that is the inner product $\langle \nu, \nu\times h\rangle =0$ for continuous $h$. Hence, $g$ is usually assumed to be a tangential vector field, for example 
\begin{equation*}
g\in L^2_{\mathrm{tan}}(\Sigma) 
= \big\{g\in L^2(\Sigma,\C^3)\colon \langle \nu, g\rangle = 0 \; \textrm{a.e.~on $\Gamma$}\big\}\,.
\end{equation*}
This boundary datum is acting as a control function and the goal is to minimize the state of the electromagnetic field at final time $T$.
On the state space 
\begin{equation} \label{e:state-space}
Y=L^2(\Omega,\C^6)
\end{equation} 
we introduce the inner product 
\begin{equation} \label{e:inner}
(y,z)_{Y} = (y_1,\ve z_1)_{L^2(\Omega)} + (y_2,\mu z_2)_{L^2(\Omega)}\,.
\end{equation}
In order to guarantee a unique solution to the initial-boundary value problem (IBVP) 
\eqref{e:maxwell}-\eqref{e:IC}-\eqref{e:BC} in $C([0,T],Y)$, we need $g \in L^2(0,T;U)$, where 
\begin{equation} \label{e:control-space}
U=H_\mathrm{tan}^{1/2}(\Gamma), 
\end{equation}
see Proposition~\ref{p:homogeneous-pbm}.
Hence, we will consider the cost functional 
\begin{equation}\label{e:cost}
J(g)=J_T(g) = \alpha \int_0^T  \|g(t,\cdot)\|^2_U\,dt + \|y(T,\cdot)\|^2_Y
\end{equation}
with $Y$ and $U$ as specified in \eqref{e:state-space} and \eqref{e:control-space}, 
respectively, and
where $\alpha>0$ is a parameter. 

\smallskip
While there is an established literature regarding the optimal control of linear evolution equations with
quadratic cost functionals, there are only few results concerning the dynamic Maxwell equations
\cite{yo17,ll02,la01}. 
The optimal boundary control of hyperbolic problems has been discussed in the classical literature (see 
\cite{bddm}, \cite{las-trig-redbooks})
and it is an accepted fact that the synthesis of the optimal solution is a much more delicate problem in the case of systems governed by hyperbolic partial differential equations (PDE) than in the case of parabolic PDE. 

Parabolic (and parabolic-like) evolutions have better regularity properties, which is helpful at many steps of the analysis. 
These play a particularly critical role in relation to the well-posedness of the Riccati equations corresponding to the optimization problem.
A rough theoretical explication in the realm of functional analysis and operator theory is the following one: the very same definition (then, the regularity) of the {\em gain} operator $B^*P(t)$ that occurs in the quadratic term of the Riccati equation (see \eqref{e:Cauchy-for-DRE}) is strongly influenced by the regularity of the operator $B^*e^{tA^*}$, where $B^*$ is intrinsically {\em unbounded} in the case of boundary or point control, while $e^{tA^*}$ is a $C_0$-semigroup in a Hilbert space $Y$.
Analytic semigroups provide the valuable feature of somehow `compensating' the unboundedness of $B$ because of its structure.
It is important to recall that this may be true as well in the case of certain problems associated with systems of coupled hyperbolic-parabolic PDE, even in the absence of analyticity of the overall semigroup; see \cite{las-cbms}, \cite{abl_2005,abl_2013}, \cite{ac-bu-uniqueness_2023} and the references therein.

The linear-quadratic problem for hyperbolic PDE is discussed at length in the monograph 
\cite[Vol.~II, Ch.~8-10]{las-trig-redbooks}, under a distinctive assumption on the linear operators
$A$ and $B$ that describe the controlled dynamics $y'=Ay+Bg$.
This inequality -- recorded explicitly and discussed in Remark~\ref{r:trace-regularity}, see \eqref{e:admissibility} --, corresponds to a (boundary) regularity estimate, specific to either PDE problem; this explains why it is termed in the literature ``abstract trace regularity'' condition (or
``admissibility condition'').
Hyperbolic systems of first-order are discussed in some generality. However, in the chosen illustration the boundary condition satisfies the Kreiss-Sakamoto condition and the (lateral) boundary is non-characteristic.
Also, a recurring feature is that the observation/weighting operators $C$ and $G$ that occur in the quadratic functional to be minimized, that is (in its general abstract form) 
\begin{equation} \label{e:abstract-general-cost}
J(g) = \int_0^T \big(\|Cy(t)\|^2_Z + \|g(t,\cdot)\|^2_U\,dt\big) + \|Gy(T)\|^2_Y\,,
\end{equation}
are assumed to be suitably smoothing; see \cite[Vol.~II, Ch.~8-10]{las-trig-redbooks}.
This allows to establish the well-posedness of the differential Riccati equation thereby achieving
a full synthesis\footnote{It will be made clear during the analysis that the fact that $B^*P(t)$ is well-defined on the optimal evolution and thus giving sense to the feedback formula is one thing; the property that $B^*P(t)$ is bounded on the state space $Y$ 
(or at least on a dense subset of it) is another.} of the optimal control. 
(To understand this point, one needs once again to take into account how the optimal cost operator $P(t)$ depends on $e^{tA^*}C^*C$ and $e^{tA^*}G^*G$; see e.g. the formula \eqref{e:riccati-op-general} (wherein $C=0$). 
The gain operator $B^*P(t)$ will depend on $B^*e^{tA^*}C^*C$ and $B^*e^{tA^*}G^*G$ accordingly.)
  
\smallskip
The Maxwell system under consideration does not fit into the general framework for two reasons. The boundary condition is conservative in the sense of \cite[Definition 1.1]{el12} and does not satisfy the Kreiss-Sakamoto condition. 
Furthermore, our weighting operator in \eqref{e:cost} is not smoothing: we aim at allowing, e.g., $G=I$. 
Indeed, the consideration of a functional where the state is observed in the energy space topology is a physically realistic choice.

The Kreiss-Sakamoto condition is necessary and sufficient for the well-posedness of the initial-boundary value problem in $C([0,T],L^2(\Omega))$ with initial data in $L^2(\Omega)$ and boundary data in $L^2(\Sigma)$, and for the $L^2$-regularity of the traces of the solution on the lateral boundary \cite{kr70,ra71,chazarainpiriou82}. Furthermore, under these conditions there is a regularity theory: more regular initial data and boundary data result in more regular solutions and traces \cite{rauch72}. The two works by Lagnese \cite{la01} and Lagnese and Leugering \cite{ll02} discuss the optimal control of Maxwell's equation by means of the so-called impedance boundary condition which satisfies the Kreiss-Sakamoto condition. The work by Yousept \cite{yo17} uses an internal control; in this case no regularity issues arise.   

In order to obtain well-posedness for the initial-boundary value problem \eqref{e:maxwell}-\eqref{e:BC} in $C([0,T],L^2(\Omega))$, one needs extra space regularity of the boundary data and does not obtain good trace regularity \cite{el12}. The regularity theory in this case is also not a simple matter and has been completed only rather recently \cite{s22}. We summarize these results below in Proposition \ref{p:nonhomogeneous-pbm}. 
With this result in hand we discuss the optimal control problem following largely the established path in the literature. 
Starting with the existence and the uniqueness of the {\em open-loop} optimal solution, we move to the feedback control in {\em closed-loop} form and discuss the Riccati theory. 
The outcomes are detailed in Theorem~\ref{t:main} below:
a feedback representation of the unique optimal control is valid, with a gain operator that is well-defined on the optimal state; however, in the absence of well-posedness of the Riccati equation -- which would bring about its unique solution $P(t)$, entering into the said formula -- the optimal control (as well as $P(t)$) is attained only via an approximation procedure.    

Lastly, we confine ourselves to a Maxwell system with zero conductivity\footnote{ This means that the region $\Omega$ is filled with material which is not a conductor.}.
As we will see later, in this case the work of Flandoli \cite{flandoli_1987} allows to
obtain $P(t)$ as a result of the well-posedeness of the dual Riccati equation.
This is true for quadratic functionals to be minimized in the general form \eqref{e:abstract-general-cost}, while in particular in the case of our functional \eqref{e:cost} the optimal solution is determined even more straightforwardly.  

While our focus is on Maxwell equations, our approach can serve as an example of other optimal control problems for hyperbolic systems subject to a boundary control action exerted through a conservative boundary condition. 
We expect that the optimal control of the linear elastic wave equations with a control of the traction on the boundary can be analyzed in a similar fashion. 

\medskip


\subsection{Main results}
For expository purposes, the principal outcomes of this work are gathered (in this introduction) in a unique result, that is Theorem~\ref{t:main} below, pertaining to the Maxwell system \eqref{e:maxwell} with non-trivial matrix $\sigma$.

We introduce the family of functionals
\begin{equation}\label{e:cost_s}
J_s(g)=\alpha \int_s^T  \|g(t,\cdot)\|^2_U\,dt + \|y(T,\cdot)\|^2_Y
\end{equation}
with $s$ varying in $[0,T)$; the relative optimal control problem is formulated accordingly.


\begin{problem}[\bf Parametric optimal control problem] \label{p:problem_s}
Given $y_0(\cdot)\in Y$, minimize the functional \eqref{e:cost_s} overall $g\in L^2(s,T;U)$, 
with $y(t,x)$ subject to \eqref{e:maxwell}-\eqref{e:BC} corresponding to a boundary control function $g(t,x)$ and to a datum $y_0(\cdot)$ at the initial time $s\in [0,T)$.
\end{problem}

 
\begin{theorem}[Main result] \label{t:main}
With reference to Problem~\ref{p:problem_s}, the following statements hold true for any $s\in [0,T)$.

\begin{enumerate}

\item[\bf S1.] 
For each $y_0\in Y$ there exists a unique optimal pair $(\hat{g}(\cdot,s,y_0),\hat{y}(\cdot,s;y_0))$
which satisfies 
\begin{equation} \label{e:optimal-pair}
\hat{g}(\cdot,s,y_0)\in L^2(s,T;U)\,, \quad \hat{y}(\cdot,s;y_0)\in C([s,T],Y)\,.
\end{equation}

\smallskip

\item[\bf S2.] 
Given $t\in [s,T]$, the linear bounded operator $\Phi(t,s)\colon Y \longrightarrow Y$ defined by 
\begin{equation} \label{e:evolution-map} 
\Phi(t,s)y_0 :=\hat{y}(\cdot,s,y_0)
\end{equation}
is an evolution operator, namely, it satisfies the transition property
\begin{equation*}
\Phi(t,t)=I\,, \quad 
\Phi(t,s)=\Phi(t,\tau)\Phi(\tau,s) \quad \textrm{for $s\le \tau\le t\le T$.}
\end{equation*}

\smallskip

\item[\bf S3.]
There exists a linear bounded operator $P(s)$ on $Y$ defined in terms of the optimal evolution and of the data of the problem (see \eqref{e:riccati-op-general}, which reduces to \eqref{e:riccati-operator} when $G=I$), such that the optimal cost is given by  
\begin{equation} \label{e:optimal-cost_1}
J_s(\hat{g})=\big(P(s)y_0,y_0\big)_Y\,.
\end{equation}
$P(t)$ is a non-negative self-adjoint operator in $Y$, for any $t\in [0,T]$;
it belongs to $C_s([0,T],Y)$, to wit, the map $z \longrightarrow P(t)z$ is continuous for each $z\in Y$.

\smallskip

\item[\bf S4.] 
The optimal control admits the following feedback representation
\begin{equation} \label{e:feedback}
\hat{g}(t,s,y_0)=-B^*P(t)\hat{y}(t,s,y_0)\,,
\end{equation}
with $\hat{y}(\cdot)$ given by \eqref{e:evolution-map}.

\smallskip

\item[\bf S5.]  
The operator $P(t)$ in {\bf S3.} is a solution to the Cauchy problem associated with the differential Riccati equation
\begin{equation} \label{e:Cauchy-for-DRE}
\hspace{6.5mm}
\begin{cases}
\frac{d}{dt}\big(P(t)x,y\big)_Y + \big(P(t)x,Ay\big)_Y + \big(Ax,P(t)y\big)_Y   
\\
\myspace -\big(B^*P(t)x,B^*P(t)y\big)_U=0\,, \quad x,y \in \cD(A)\,, \; t\in [0,T)
\\[1mm]
P(T)=I 
\end{cases}
\end{equation}
in a generalized sense (only): to wit, there exists a sequence of operators $\{P_n(\cdot)\}_n\in \cL(Y)$ 
such that

\medskip 
\begin{itemize}
\item
$P_n(t)$ solves \eqref{e:Cauchy-for-DRE}, with the gain operator $B^*P_n(t)$ which is bounded (for all $n$),
\item
the asymptotic result  
\begin{equation*}
\|P_n(\cdot)z-P(\cdot)z\|_Y\longrightarrow 0 \; \textrm{in $C([0,T],Y)$, for all $z\in Y$}
\end{equation*}
holds.
\end{itemize}

\end{enumerate}

\end{theorem}

\smallskip


If $\sigma=0$ in the Maxwell's system \eqref{e:maxwell}, while all of the statements in Theorem~\ref{t:main} are still valid, we may actually make use of the work
\cite{flandoli_1987} where the optimal cost (or Riccati) operator $P(t)$ is shown to be the inverse of 
the unique solution $Q(t)$ to the Cauchy problem associated with the dual Riccati equation,
which reads (in its full form) as 
\begin{equation} \label{e:Cauchy-for-dualRE}
\begin{cases}
\frac{d}{dt}\big(Q(t)x,y\big)_Y = \big(Q(t)x,A^*y\big)_Y + \big(A^*x,Q(t)y\big)_Y 
-\big(B^*x,B^*y\big)_U
\\[2mm]
\hspace{3cm}+ \big(CQ(t)x,CQ(t)y\big)_Y =0\,, \quad x,y \in \cD(A^*)\,, \; t\in [0,T)
\\[1mm]
Q(T)=(G^*G)^{-1}\,.
\end{cases}
\end{equation}
Furthermore, as we will see later, the work \cite{flandoli_1987} provides a distinct {\em open-loop} representation formula for the optimal solution.

\begin{remark}
\begin{rm}  
We note here that the well-posedeness of \eqref{e:Cauchy-for-dualRE} can be established more readily owing to the fact that the operator $C$ which occurs in the quadratic term is bounded; see e.g. \cite[Theorem~2.1.1]{flandoli_1987} or \cite[Theorem~9.6.2.4]{las-trig-redbooks}.
We emphasize at the outset that in the case where $C=0$, the equation in \eqref{e:Cauchy-for-dualRE} is not even quadratic. 
\end{rm}
\end{remark}

The special case $\sigma=0$ will be discussed in greater detail in Section~\ref{s:zero-conductivity};
the corresponding findings are collected in Proposition~\ref{p:from-Flandoli}.
For the reader's convenience and the sake of completeness we will provide a statement which covers
cost functionals that display non-trivial (and non-smoothing) weighting operators $C$ and $G$ 
(see \eqref{e:abstract-general-cost}), while the weighting operator $G$ is further
assumed to be invertible.
Of course the above encompasses the specific case of interest $C=0$, $G=I$; however, the respective result will be explicitly specified.


\subsection{An outline of the paper}
In the next section we begin by setting up the well-posedness of the IBVP \eqref{e:maxwell}-\eqref{e:IC}-\eqref{e:BC} from the natural PDE perspective first (see Proposition~\ref{p:nonhomogeneous-pbm}), and  next from a functional-analytic (operator-theoretic) one; see the sections~\ref{ss:A} and \ref{ss:B}).
Both viewpoints are key to our analysis.
 
The optimal control problem is considered in Section~\ref{s:OCP}, where we show that the sought outcomes stated as S1.-S4. of Theorem~\ref{t:main} hold.
Although the line of argument is classical, there is some step where the low regularity of the weighting operator $G$ demands a careful analysis. 
What is left unachieved is the property that the gain operator $B^*P(t)$ ($P(t)$ being the optimal cost operator) -- which is well-defined on the optimal evolution, thus giving sense to the feedback
formula in S4. -- is bounded on the state space $Y$ (or at least on a dense subset of it).

Thus, in Section~\ref{s:approximating} we introduce a sequence of optimization problems for which the optimal synthesis holds true via proper differential Riccati equations, with the Riccati operator as well as the optimal solution (of the original problem) recovered in the limit.
This is accomplished in Proposition~\ref{e:extended-sln}, thereby establishing the statement S5. of Theorem~\ref{t:main}.  

Finally, Section~\ref{s:zero-conductivity} focuses on the same problem under the assumption that
the conductivity is zero. 
In this case, while all the statements of Theorem~\ref{t:main} remain true,
a distinct sharper result is valid, that is Proposition~\ref{p:from-Flandoli}. 

An additional proof of the valuable fact that $P(t)$ is an isomorphism is provided {\em a priori}
in Proposition~\ref{p:isomorphism}, independently of the results in \cite{flandoli_1987}. 


\section{On well-posedness of Maxwell's equations} \label{s:setup}
In this section we discuss and assess the well-posedness of the IBVP \eqref{e:maxwell}-\eqref{e:IC}-\eqref{e:BC} (with non-trivial initial and boundary data) in a chosen functional setting, as this is a key prerequisite to the study of the associated optimal control problems.
We first recall a critical well-posedness result, whose proof was carried out within the realm of hyperbolic systems theory and entailed para-differential calculus.
This result is then interpreted within a functional-analytic framework 
via operator semigroup theory. 
Since both perspectives are useful in our investigation, the (respective) needed elements are introduced and detailed below.

\subsection{A major well-posedness result. Weak solutions}
The boundary condition \eqref{e:BC} is conservative in the sense of \cite[Definition 1.1]{el12}. 
Hence, we have the following well-posedness result for the IBVP \eqref{e:maxwell}-\eqref{e:IC}-\eqref{e:BC} which is an application of Theorem~1.4 in \cite{el12} and Theorem~1.1 in \cite{s22}.
 
 
 
\begin{proposition} \label{p:nonhomogeneous-pbm}
For $g\in L^2(0,T;H_{\mathrm{tan}}^{1/2}(\Gamma))$ and $y_0\in Y$ there exists a unique weak solution $y\in C([0,T],Y)$ such that 
\begin{equation*}
\nu\times e\big|_{\Sigma} \in L^2(0,T;H_{\mathrm{tan}}^{-1/2}(\Gamma))\,.
\end{equation*}  
For $g\in L^2(0,T;H_{\mathrm{tan}}^{3/2}(\Gamma))\cap H^1(0,T; H_{\mathrm{tan}}^{1/2}(\Gamma))$ and $y^0\in H^1(\Omega,\C^6)$ satisfying the compatibility condition $g(0) = \nu\times y_0$ on $\Gamma$ there exists a unique differentiable solution $y \in C([0,T],H^1(\Omega,\C^6))$. 
\end{proposition} 

\smallskip
The weak solution satisfies the identity
\begin{multline*}
(e,\ve \partial_t z_1 - \nabla\times z_2-\sigma z_1)_{L^2(Q)} 
+ (h,\mu \partial_t z_2 + \nabla\times z_1)_{L^2(Q)} 
\\
\qquad =(e(T),\ve z_1(T))_{L^2(\Omega)} - (e_0,\ve z_1(0))_{L^2(\Omega)} 
+ (h(T),\mu z_2(T))_{L^2(\Omega)}
- (h_0,\mu z_2(0))_{L^2(\Omega)} 
\\
-(g,z_1)_{L^2(\Sigma)} - (e,\nu\times z_2)_{L^2(\Sigma)}
\end{multline*}
for all $z=(z_1,z_2)\in H^1(Q,\C^6)$.
With $y=(e,h)$, recalling the inner product \eqref{e:inner} in $Y$, the above identity for the weak solution is rewritten as 
\begin{multline}\label{5}
(e,\ve \partial_t z_1 - \nabla\times z_2-\sigma z_1)_{L^2(Q)} 
+ (h,\mu \partial_t z_2 + \nabla\times z_1)_{L^2(Q)} 
\\
=(y(T), z(T))_{Y} - (y_0,z(0))_{Y}
- (g,z_1)_{L^2(\Sigma)} - (e,\nu\times z_2)_{L^2(\Sigma)}\,.
\end{multline} 

\smallskip

\subsubsection{The operator underlying the free dynamics. Mild and strict solutions}
\label{ss:A} 
We also discuss briefly the IBVP with homogeneous boundary data, i.e. with $g\equiv 0$. 
Then it is convenient to use semigroup theory in order to obtain solutions of higher regularity.
In the state space $Y$ (recall \eqref{e:state-space}) we define the following unbounded operator $A$:
\begin{equation} \label{e:A}
\hspace{-1cm}
\begin{cases}
\cD(A)= H(\mathrm{curl},\Omega)\times H_0(\mathrm{curl},\Omega)   
\\[1mm]
Ay = \begin{bmatrix}
\ve^{-1}\nabla\times y_2 -\ve^{-1}\sigma y_1
\\[1mm]
-\mu^{-1}\nabla\times y_1
\end{bmatrix}\,,
\end{cases}
\end{equation}   
where
\[
H(\mathrm{curl},\Omega) = \big\{ u \in L^2(\Omega,\C^3)\colon  
\nabla \times u \in L^2(\Omega,\C^3)\big\}
\]
with the norm $\|u\|_{H(\mathrm{curl},\Omega)}^2 = \|u\|^2_\Omega + \|\nabla\times u\|^2_\Omega$. 
An important fact is that the trace mappings 
\begin{equation*}
u \longmapsto \nu\times u\big|_{\Gamma} \quad \textrm{and} 
\quad u \longmapsto u_\mathrm{tan} \big|_{\Gamma}
\end{equation*}
which are defined for continuous functions extend to continuous linear operators from $H(\mathrm{curl},\Omega)$ onto 
\[
H^{-1/2}(\mathrm{div}_\Gamma) =  \big\{w \in H_{\mathrm{tan}}^{-1/2}(\Gamma)\colon
\mathrm{div}_\Gamma w\in H^{-1/2}(\Gamma)\}\,,
\]  
and from $H(\mathrm{curl},\Omega)$ onto 
\[
H^{-1/2}(\mathrm{curl}_\Gamma) = \{ w \in H_{\mathrm{tan}}^{-1/2}(\Gamma)\colon 
\mathrm{curl}_\Gamma w\in H^{-1/2}(\Gamma)\}\,,
\]
respectively.   
The space $H_0(\mathrm{curl},\Omega)$ is a closed subspace of $H(\mathrm{curl},\Omega)$ defined by 
\[
H_0(\mathrm{curl},\Omega) = \{ u\in H(\mathrm{curl},\Omega)\colon \nu\times u =0 \;
\textrm{on $\Gamma$}\}\,.
\] 
For a good introduction into these function spaces we refer to the book by Cessenat \cite{cessenat96}.
Given $z\in \cD(A)$ we denote the tangential trace on $\Gamma$ of the first component function $z_1$  by $z_{1,\mathrm{tan}}$.  

The adjoint operator $A^*$ is computed readily: we have 
\begin{equation} \label{e:A-star}
A^* z = \begin{bmatrix}
-\ve^{-1}\nabla\times z_2 -\ve^{-1}\sigma z_1 
\\[1mm]
\mu^{-1} \nabla\times z_1
\end{bmatrix}\,,
\end{equation}
with $\cD(A^*)\equiv\cD(A)=H(\mathrm{curl},\Omega)\times H_0(\mathrm{curl},\Omega)$.

Semigroup theory allows to establish the well-posedness of the Cauchy problem
\begin{equation} \label{e:cauchy}
\begin{cases}
y'=Ay & t>0
\\
y(0)=y_0 &
\end{cases}
\end{equation}
via the Lumer-Phillips Theorem.

Although the result is probably well known, for the readers' convenience we collect the pertinent statements in the following proposition, referring~e.g.~to \cite[Proposition\,1]{el19} for the proof.
Recall that the solutions corresponding to initial data $y_0$ in the state space $Y$ are intended in a {\em mild} sense.
When $y_0\in \cD(A)$, then the corresponding solution $y(t)=e^{tA}y_0$ is {\em strict}; 
to wit, it indeed belongs to $C([0,\infty),D(A))\cap C^1([0,\infty),Y)$.


\begin{proposition} \label{p:homogeneous-pbm}
The operator $A$ defined by \eqref{e:A} is the generator of a strongly continuous contraction semigroup $S_t=e^{tA}$ in $Y$, $t\ge 0$.
Then, for any $y_0 \in Y$ there exists a unique mild solution $y(t) = e^{tA}y_0$ to the Cauchy problem 
\eqref{e:cauchy} (and hence to the IBVP problem \eqref{e:maxwell}-\eqref{e:IC}-\eqref{e:BC} with $g\equiv 0$); then, $y\in C([0,\infty),Y)$.

If $y_0 \in \cD(A)$ the unique solution is strict, namely,
\[
y \in C([0,\infty),\cD(A))\cap C^1([0,\infty),Y)\,.
\]  
\end{proposition}

\smallskip

\subsection{The input-to-state map, the control operator. Direct inequality}
\label{ss:admissibility}
Consider the bounded operator $L$ which maps the boundary data $g$ to the solution of the IBVP \eqref{e:maxwell}-\eqref{e:IC}-\eqref{e:BC} with initial data zero.
We denote by 
\[
L_T\colon L^2(0,T;U) \longrightarrow Y
\]
the bounded operator which maps the boundary data $g$ to $(Lg)(T)$, that is the solution of the IBVP \eqref{e:maxwell}-\eqref{e:IC}-\eqref{e:BC} with initial data zero at time $T$; then $L_Tg:= (Lg)(T)$.
In order to describe its adjoint $L_T^*$ we fix a scalar product in $U=H^{1/2}_\mathrm{tan}(\Gamma)$. In the following we will set
\begin{equation}\label{e:product}
(f,g)_U = (\sqrt{-\Delta_\Gamma}f,g)_{L^2(\Gamma)}\,,
\end{equation}
where $-\Delta_\Gamma$ denotes the Laplace-Beltrami operator on $\Gamma$. 
The Laplace-Beltrami operator is known to be positive and self-adjoint, hence its square root
$\sqrt{-\Delta_\Gamma}$ is given by the functional calculus for these operators. 


\begin{proposition}\label{p:L_T^*}
The operator $L_T^*\colon Y \to L^2(0,T;U)$ is given by the formula
\begin{equation} \label{e:L_T-adjoint}
L_T^* p = (-\Delta_\Gamma)^{-1/2}[e^{A^\ast(T-\cdot )}p]_{1,\mathrm{tan}}\,, \qquad p\in Y\,.
\end{equation}

\end{proposition}

\begin{proof}
Recall that $y_0=(e_0,h_0)=0$. From the definition of the weak solution we see that 
\begin{equation*} 
(L_T g, p)_{Y}=(y(T),p)_{Y}= (g,z_1)_{L^2(\Sigma)}
\end{equation*}
whenever $z\in H^1(Q,\C^6)$ satisfies the system 
\[
\ve \partial_t z_1 - \nabla\times z_2 - \sigma z_1 
=  \mu \partial_t z_2 + \nabla\times z_1 = 0 \quad \textrm{in $Q$,}
\]
with final data $z(T)=p$ and the boundary condition $\nu\times z_2=0$ on $\Sigma$. 
Using the adjoint semigroup of $e^{tA}$ we can write $z(t) = e^{(T-t)A^*}p$.
Of course, Propositions~\ref{p:nonhomogeneous-pbm} and \ref{p:homogeneous-pbm} apply to this system.  
Hence, we have shown that 
\[
(L_T g, p)_{Y}  =(y(T),p)_{Y} = (g,[e^{A^\ast (T-\cdot)}p]_{1,\mathrm{tan}})_{L^2(\Sigma)}
\quad  \forall p\in H^1(\Omega,\C^6)\cap \cD(A)\,.
\]
By density this formula is also valid for all $p\in Y$ and the formula \eqref{e:L_T-adjoint} 
for the adjoint operator is proved in view of \eqref{e:product}.
\end{proof}


\subsubsection{The Green map}
We define a linear operator $\cG\colon H^{1/2}_{\mathrm{tan}}(\Gamma) \to H(\mathrm{curl},\Omega)^2$ by setting $\cG g= y$ where $y=(y_1,y_2)$ is the solution to the boundary value problem
\[
 \begin{split}
  \ve y_1 - \nabla \times y_2 +\sigma y_1 & = 0 
  \\ 
  \mu y_2 + \nabla\times y_1 &= 0 
\end{split}
\qquad \mbox{ in } \Omega 
\]  
with
\[
 \nu\times y_2 = g \quad \mbox{ in } \Gamma\;.  
\]
\begin{proposition}
The operator $\cG\colon H^{1/2}_{\mathrm{tan}}(\Gamma) \to H(\mathrm{curl},\Omega)^2$ is continuous.
\end{proposition}

\begin{proof}
On $H(\mathrm{curl},\Omega)$ we introduce the continuous bilinear form
\[
  a(y_1,z_1) = \int_\Omega \langle \mu^{-1}\nabla\times y_1,\nabla\times z_1\rangle + \langle (\ve+\sigma) y_1,z_1\rangle 
 \,dx
\]
and the continuous linear functional 
\[
 l(z_1) =  \int_\Gamma \langle g,z_{1,\mathrm{tan}}\rangle \,d\Sigma\;.
\] 
Since the bilinear form is coercive, that is $a(y,y) \gtrsim \|y_1\|^2_{H(\mathrm{curl},\Omega)}$, by the Lax-Milgram lemma the exists a unique solution $y_1 \in H(\mathrm{curl},\Omega)$ such that 
\begin{equation*}
a(y_1,z) = l(z) \quad \textrm{for all $z\in H(\mathrm{curl},\Omega)$;}
\end{equation*}
with this very same $y_1$, define
\[
y_2 = -\mu^{-1}\nabla\times y_1\,.
\]
Then $y_2\in Y$ and choosing $z_1\in C_0^1(\Omega,\R^3)$ we see that $y_1$ is a weak solution to 
\[
 \ve y_1 - \nabla\times y_2 +\sigma y_1= 0\;,
\]
whence $y_2\in H(\mathrm{curl},\Omega)$.

The only thing remaining is to verify the boundary conditions. Suppose $z_1 \in H^1(\Omega,\C^3)$. Then 
\[
 -\int_\Omega \langle  y_2,\nabla\times z_1\rangle - \langle (\ve+\sigma) y_1,z_1\rangle  \,dx =
 \int_\Gamma \langle g,z_{1,\mathrm{tan}}\rangle \,d\sigma
\]
and after performing integration by parts in the first integral and using $(\ve+\sigma)y_1 - \nabla\times y_2=0$,  we see that 
\[
 \int_\Gamma \langle \nu \times y_2,z_{1,\mathrm{tan}}\rangle \,d\sigma =
 \int_\Gamma \langle g,z_{1,\mathrm{tan}}\rangle \,d\sigma\;,
\] 
which verifies the boundary condition.
\end{proof}

 
\subsubsection{The control operator $B$} \label{ss:B}
Exploiting the density of $\cD(A)$ in $Y$, we can extend the operator $A$ as a continuous operator from $Y$ into $[\cD(A^*)]'\equiv[\cD(A)]'$ by defining $A_{\rm ext}y\in [\cD(A)]'$ for $y\in Y$ by setting
\[
(A_{\rm ext}y,z)_Y = (y,A^*z)_Y \quad z\in  \cD(A^*)\,.
\] 
The extended operator is an unbounded operator on the Hilbert space $[\cD(A)]'$, with $\cD(A_{\rm ext}) = Y$; it is the generator of a strongly continuous contraction semigroup 
$e^{tA_{\rm ext}}$ on $[\cD(A)]'$, which is an extension of the semigroup $e^{tA}$ (and will be often denoted by the same symbol). 

\smallskip
The control operator $B$ is defined by 
\begin{equation} \label{e:B}
B = (I-A)\cG
\end{equation}
and maps continuously from $U=H^{1/2}_{\mathrm{tan}}(\Gamma)$ into $[\cD(A)]'$. 
Accordingly, its adjoint $B^*$ is a continuous operator from $\cD(A)$ into $H_{\mathrm{tan}}^{1/2}(\Gamma)$. 

Since $(I-A(x,\partial))\cG g=0$, where $A(x,\partial)$ denotes the operator which acts like $A$ but on the  maximal domain $H(\mathrm{curl},\Omega)^2$, the IBVP \eqref{e:maxwell}-\eqref{e:IC}-\eqref{e:BC} can be rewritten as 
\[
\begin{cases}
y_t = (A-I)(y - \cG g) + y  &  \textrm{in $Q$}
\\
y(0,\cdot)=y_0(\cdot)  &  \textrm{in $\Omega$}
\\
\nu\times (y - \cG g) =0  &  \textrm{on $\Sigma$.} 
\end{cases}
\]  
Using the extension of $A$ discussed above, we find that the original IBVP \eqref{e:maxwell}-\eqref{e:IC}-\eqref{e:BC} can be recasted as the Cauchy problem  
\begin{equation} \label{e:full-cauchy}
\begin{cases}
y'=Ay +Bg & 0<t<T
\\
y(0)=y_0\in Y\,, &
\end{cases}
\end{equation}
the control system $y'=Ay +Bg$ initially is meant in the space $[\cD(A)]'$.
 
As $g\in L^2(0,T;U)$, we have $Bg \in L^2((0,T),[\cD(A)]')$ and hence the mild solution of problem \eqref{e:full-cauchy}, that is 
\begin{equation} \label{e:full-mild}
y(t) = e^{tA}y_0+\underbrace{\int_0^t e^{(t-s)A}Bg(s)\,ds}_{ (Lg)(t)}\,,
\end{equation}
satisfies (a priori) $y \in C([0,T],[\cD(A)]')$.
However, in view of Proposition \ref{p:nonhomogeneous-pbm} we know that the convolution term
$(Lg)(\cdot)$ possesses an actually better (spatial) regularity -- which is then inherited by $y(\cdot)$. This property is central to the study of the associated optimal control problems.
For clarity, we render explicit this issue and expand briefly on it in the following

\begin{remark} \label{r:trace-regularity}
\begin{rm}
With the dynamics and control operators $A$ and $B$ defined by \eqref{e:A} and \eqref{e:B}, respectively, let $L$ be the {\em input-to-state map} defined by
\begin{equation*}
L\colon L^2(0,T;U)\ni g \longrightarrow (Lg)(\cdot):=\int_0^\cdot e^{(\cdot-s)A} Bg(s)\,ds\,.
\end{equation*}  
Thanks to Proposition~\ref{p:nonhomogeneous-pbm}, we know that 
\begin{equation} \label{e:continuity}
L \in \cL(L^2(0,T;U),C([0,T],Y))\,.
\end{equation}
Consequently, the estimate
\begin{equation} \label{e:admissibility}
\exists\, c_T>0\colon \quad \int_0^T \|B^*e^{tA^*} x\|_{U}^2\,dt\le c_T \|x\|_Y^2 \qquad \forall x\in \cD(A)
\end{equation}
(known in the literature as ``abstract trace regularity'' or ``admissibility condition'') holds true; it corresponds to the property
 
\begin{itemize}
\item[]
{\em the operator $B^*e^{\cdot A^*}$ admits a continuous extension from $Y$ to $L^2(0,T;U)$}
\end{itemize}
For further details and references we direct the reader to \cite[Section 7.1]{las-trig-redbooks};
the equivalence between the conditions \eqref{e:admissibility} and \eqref{e:continuity}, in particular, is proved in Theorem~7.2.1 (set $p=q=2$ and identify $Y^*$ with $Y$ as well as $U^*$ with $U$).
\end{rm}
\end{remark}
\smallskip
The action of $B^*$ as a trace operator is made precise in the following result.  

\begin{proposition}
For $z \in \cD(A)$ we have 
$B^* z = (-\Delta_\Gamma)^{-1/2}z_{1,\mathrm{tan}}$.
\end{proposition}

\begin{proof}
We start with
\[
(B^* z, g)_\Gamma = ((I - A^*)z, y)_Y\,,
\]   
where $y = \cG g$. 
Then using the definition \eqref{e:A-star} of $A^*$ and integration by parts we compute
 \[
  \begin{split}
   ( ( I - {A}^\ast)z, y)_Y  =& \int_\Omega
   \langle \ve \alpha z_1 +\nabla\times z_2+\sigma z_1,y_1 \rangle
   +\langle \mu \alpha z_2 -\nabla\times z_1,y_2\rangle\,dx  
   \\ =& 
   \int_\Omega
    \langle z_1 ,\ve \alpha y_1 -\nabla\times y_2+\sigma y_1\rangle
    + \langle   z_2, \mu \alpha y_2 + \nabla \times y_1 \rangle \, dx
    \\ &\myspace + \int_{\Gamma}\langle \nu\times z_2,y_1\rangle 
           - \langle \nu\times z_1,y_2\rangle\,dS   
    \\=& ( z_{1,\mathrm{tan}}, g )_\Gamma =
    (\sqrt{-\Delta_\Gamma} z_{1,\mathrm{tan}}, g )_U
     \end{split}
\]   
where we made use of $y = \cG g$. 
\end{proof}


\section{The optimal control problem: The case $U=H^{1/2}_\mathrm{tan}(\Gamma)$} \label{s:OCP}
For expository purposes, the main results of this paper have been collected 
in aggregate form in the introduction, as statements of Theorem~\ref{t:main}.
This Section is devoted to the proofs of the statements S1, S2, S3. and S4. of Theorem~\ref{t:main}.
Starting from the existence of a unique {\em open-loop} optimal solution to the optimization problem, in this section we provide the details of the steps bringing about the {\em closed-loop} representation of the optimal control. 
Although our line of argument retraces a well-established one carried out in literature,
we provide in particular an accurate version of a proof that is specific to the Bolza problem for hyperbolic (or hyperbolic-like) equations, in the absence of smoothing observations.
This pertains to the regularity of the evolution map $\Phi(t,s)$ (defined below in \eqref{e:opt-state}) with respect to the parameter $s$, which is key to the regularity of the optimal cost operator $P(s)$.


\subsection{The open-loop optimal solution}
The unique mild solution to the abstract Cauchy problem \eqref{e:full-cauchy} at final time $T$ can be written in the form 
\[
y(T) =  e^{TA}y_0 + L_T g
\]
for $y_0 \in Y$ and $g\in L^2(0,T;U)$.  
For $y_0\in Y$ we have 
\[
J(g) =  \alpha \int_0^T  \|g(t)\|^2_U\,dt + \|y(T)\|^2_Y\,. 
\] 
Using the adjoint $L_T^*\colon L^2(\Omega) \to L^2(0,T;U)$ we have for $g\in L^2(0,T;U)$ and
$y_0 \in Y$
\[
 \begin{split}
J(g) &= \alpha\int_0^T (g,g)_U\,dt  + (e^{TA}y_0 + L_T g,e^{TA}y_0 + L_T g)_Y 
\\
&= \int_0^T ([\alpha I + L_T^*L_T] g,g)_U + 2\Re \int_0^T (L_T^*e^{TA}y_0,g)_U\,dt 
+\|e^{TA}y_0\|^2_Y\;.
\end{split}
\]
Hence, with the choice
\begin{equation}\label{e:open-loop}
\hat{g}(\cdot) = - \Lambda_T^{-1}L_T^\ast e^{TA}y_0 \in L^2(0,T;U)\;,
\end{equation}
where $\Lambda_T = \alpha I + L_T^\ast L_T$ is positive self-adjoint, and continuously invertible 
with $\|\Lambda^{-1}_T\|\le \alpha^{-1}$, we have 
\[
J(g) =\int_0^T  (\Lambda_T g,g)_U\,dt - 2\Re \int_0^T (\Lambda_T \hat{g},g)_U\,dt
+\|e^{{A}T}y_0\|^2_Y 
\]   
and we observe that $J$ attains its unique minimum at $\hat{g}$ since 
\begin{equation*}
\begin{split}
J(g)-J(\hat{g})
&= \int_0^T (\Lambda_T g,g)_U\,dt - 2\Re\int_0^T ( \Lambda_T \hat{g},g)_U\,dt  
+ \int_0^T ( \Lambda_T \hat{g},\hat{g})_U\,dt  
\\   
& =\int_0^T  (\Lambda_T [g-\hat{g}],g-\hat{g})_U\,dt > 0  
\end{split}
\end{equation*}
for all $g\neq \hat{g}$. 


\subsection{Closed-loop form of the optimal control}
Let $0\le s< T$. 
We introduce the parametrized family of optimal control problems, where the goal is now to minimize the functional   
\begin{equation} \label{e:family-cost}
J_s(g)=J_{T,s}(g) = \alpha \int_s^T \|g(t)\|^2_U +  \|y(T,s,y_0)\|_Y^2
\end{equation} 
overall the control functions $g\in L^2(s,T;U)$, where $y(t,s,y_0)$ is the mild solution to the Cauchy problem 
\begin{equation} \label{e:cauchy-family}
\begin{cases}
y'=Ay +Bg & s<t<T
\\
y(s)=y_0
\end{cases}
\end{equation}
(the initial time $s\in [0,T)$ is clearly the parameter).
Explicitly, the state at time $t$ is given by the formula
\begin{equation} \label{e:mild-family}
y(t,s,y_0) = e^{(t-s)A}y_0 + \int_s^t e^{A(t-\tau)}Bg(\tau)\,d\tau =: e^{(t-s)A}x + (L_s g)(t)\,.
\end{equation} 
The notation 
\begin{itemize}
\item
$L_s$ for the linear operator that maps the boundary data $g\in L^2(s,T;U)$ to the solution 
$(L_s g)(\cdot)$ to the inital-value problem \eqref{e:cauchy-family} with $y_0=0$, whose
regularity is $L_s\in \cL(L^2(s,T;U),C([s,T],Y)$), as well as
\item
$L_{sT}$ for the continuous linear operator which maps the boundary data
in $L^2(s,T;U)$ to the solution (again, to the inital-value problem \eqref{e:cauchy-family} with $y_0=0$) at final time $T$, i.e. $L_{sT} g$ for $(L_s g)(T)$,
\end{itemize}
are consistent and standard in the literature.

The adjoint operator $L_{sT}^*\colon Y \to L^2(s,T;U)$ is given by 
\begin{equation}\label{e:L_sT-star}
(L_{sT}^* z)(t) = B^*e^{(T-t)A^*}z  \quad 0\le s\le t\le T\;.
\end{equation}
In perfect analogy to the argument which brought about the open-loop optimal control \eqref{e:open-loop}
we have a unique minimizer for $J_s$ given by 
\begin{equation}\label{e:opt-control}
\hat{g}(\cdot,s,x)  = - \Lambda_{sT}^{-1}{L}_{sT}^* e^{(T-s)A}x \in  L^2(s,T;U)
\end{equation}
for $x\in Y$, where $\Lambda_{sT} = \alpha I + {L}_{sT}^* {L}_{sT}$. 
By the continuity of the operators involved, we infer the estimate
\[
\|\hat{g}(\cdot,s,x)\|_{L^2(s,T;U)} \lesssim_T \|x\|_Y\,,
\] 
for all $s\in [0,T)$. The optimal state is then 

\begin{equation} \label{e:opt-state}
\hat{y}(t,s,x)=e^{(t-s)A}x + (L_s \hat{g}(\cdot,s,x)(t):=\Phi(t,s)x\;,
\end{equation}
where we introduce the operator $\Phi(t,s)$. Note that $\Phi(t,s)$ is linear.  

Inserting the expression \eqref{e:opt-control} into the formula \eqref{e:opt-state} gives 
\begin{equation}\label{15}
 \hat{y}(t,s,x) = e^{(t-s)A}x - [L_s\Lambda^{-1}_{sT}L_{sT}^* e^{(T-s)A}x](t) 
\end{equation}
and evaluating at $t=T$ yields
\begin{equation}\label{16}
 \hat{y}(T,s,x) =  (I-L_{sT}\Lambda^{-1}_{sT}L_{sT}^\ast) e^{(T-s)A}x\,.
\end{equation}
Using one more time the continuity of all the operators involved yields
\[
 \sup_{0\le t\le T}\|\hat{y}(t,s,x)\|_Y \lesssim_T \|x\|_Y\,,
\]
uniformly in $s$.  

Applying $\Lambda_{sT}$ to both sides in \eqref{e:opt-control} and using next
\eqref{e:opt-state} to express $e^{(T-s)A}x$ gives
\begin{equation*}\label{17}
 (\alpha + L_{sT}^\ast L_{sT})\hat{g}(\cdot,s,x)  = - L^\ast_{sT}e^{(T-s)A}x 
 =-L^\ast_{sT} \hat{y}(T,s,x)  + L_{sT}^\ast L_{sT} \hat{g}(\cdot,s,x)  
\end{equation*}
and thus 
\begin{equation}\label{e:pre-feedback}
 \hat{g}(\cdot,s,x)  = -\frac{1}{\alpha}L^\ast_{sT} \hat{y}(T,s,x)
 =-\frac{1}{\alpha}L^\ast_{sT}\Phi(T,s)x 
 =-\frac{1}{\alpha}B^\ast e^{A^\ast(T-\cdot)}\Phi(T,s)x\,.
\end{equation}

From formula \eqref{16} we infer that 
\[
 \begin{split}
 L_{sT}^\ast \hat{y}(T,s,x) = 
  (L_{sT}^\ast -L_{sT}^\ast L_{sT} \Lambda_{sT}^{-1}L_{sT}^\ast) e^{A(T-s)}x
      =\alpha \Lambda^{-1}_{sT}L_{sT}^\ast e^{A(T-s)}x
   \end{split}
\]  
  or 
\[
 L_{sT}^\ast \hat{y}(T,s,x)  =   \alpha\Lambda^{-1}_{sT}L_{sT}^\ast e^{A(T-s)}x\;.
\]  
Using this last formula in \eqref{e:opt-state} we observe that 
\begin{equation} \label{e:to-be-used-in-5}
 \hat{y}(t,s,x) =  e^{(t-s)A}x  - \frac{1}{\alpha} L_sL_{sT}^\ast \hat{y}(T,s,x)(t)\;.
\end{equation}  

\smallskip
Proving that the optimal pair $(\hat{g}(t,s,x)$,$\hat{y}(t,s,x))$ satisfies certain usual transition properties is a key step in attaining the sought feedback formula as well as in poinpointing the regularity (in time) of the optimal cost operator $P(\cdot)$.
The precise statements are given in the following Lemma. 
The corresponding proof is carried out starting from the representation \eqref{e:opt-control} of the optimal control $\hat{g}(t,s,x)$, and next moving on to the evolution map $\Phi(t,s)$ defined in \eqref{e:opt-state}; since this proof is substantially the same as the one of of \cite[Lemma 9.3.2.2, p. 779]{las-trig-redbooks} in the case $R\equiv 0$, it is omitted.


\begin{lemma}[Transition properties]
For $0\le s\le\tau\le t\le T$ we have 
 \begin{equation} \label{e:transition}
  \hat{g}(t,s,x) = \hat{g}(t,\tau,\Phi(\tau,s)x) \quad \text{for a.e. $t$,}
 \end{equation}
 and 
\begin{equation} \label{e:evolution}
 \Phi(t,s) = \Phi(t,\tau)\Phi(\tau,s)\;. 
\end{equation}

\end{lemma}

\medskip
The following result concerning the regularity of the evolution map $\Phi(t,s)$ with
respect to the variable $s$ is instrumental in achieving the regularity (in time) of the optimal cost operator $P(\cdot)$.  

Before proceeding to its statement and proof, we remark that while the Lemma below has essentially the same statement of \cite[Vol.~II, Lemma~9.3.2.3, p.~779]{las-trig-redbooks}, the proof of that result
actually utilizes the assumptions (A2) and (A3) therein, that require a quantitatively precise smoothing effects of the operators $C$ and $G$ (here, $0$ and the identity $I$, respectively); see the passage from the second to the third line at p.~783, where the estimate 9.2.8 is employed.
So while Lemma~9.3.2.3 in \cite{las-trig-redbooks} apparently holds under the only assumption (A1) -- that is \eqref{e:admissibility}, just like in the present case --, an enhanced regularity property of $\hat{g}$ is invoked in its proof which is actually unnecessary to attain the goal.


\begin{lemma}[Regularity of the evolution map]
Assume that the couple $(A,B)$ satisfies the regularity assumption \eqref{e:admissibility}.
With reference to the evolution map defined in \eqref{e:opt-state}, we have the following regularity 
\begin{equation} \label{e:transition}
\textrm{$[0,t]\ni s\longrightarrow \Phi(t,s)$ is continuous}
\end{equation}
(in the topology of $\cL(Y)$, for any given $t\le T$).
\end{lemma}

\begin{proof}
Aiming to prove continuity of $\Phi(t,s)$ at time $s$, when $s\le t$, we evaluate separately
(for $h>0$ sufficiently small)
\[
\|\Phi(t,s+h)x - \Phi(t,s)x\| \quad \textrm{and} \quad  \|\Phi(t,s-h)x - \Phi(t,s)x\|\,.
\]  
By using the evolution property \eqref{e:evolution} for $\Phi$ as well as the fact that its norm is bounded, we (routinely) first find
\begin{equation*}
\begin{split}
\|\Phi(t,s+h)x - \Phi(t,s)x\| &= \|\Phi(t,s+h)[x - \Phi(s+h,s)x]\|
\\[1mm]
&\lesssim_T\|x - \Phi(s+h,s)x\|\longrightarrow 0\,, \quad \textrm{as $h\to 0^+$,}
\end{split}
\end{equation*}
with the limit following in view of the continuity property of the mapping $\Phi(\cdot,s)$ (in other words, since the optimal state $\hat{y}(t,s,x)$ is continuous in $t$).

As for the left continuity of $s\longrightarrow \Phi(t,s)$, we compute (still with $h>0$)
\begin{equation*}
\|\Phi(t,s-h)x - \Phi(t,s)x\| = \|\Phi(t,s)[\Phi(s,s-h)x-x]\| \lesssim_T \|\Phi(s,s-h)x-x\|\,,
\end{equation*}
where we recall 
\begin{equation*}
\Phi(s,s-h)x= e^{hA}x+\big[L_{s-h}\hat{g}(\cdot,s-h,x)\big](s)\,,
\end{equation*}
so that the estimate
\begin{align}
\|\Phi(s,s-h)x-x\| &\le \|e^{hA}x-x\|_Y+\|L_{s-h}[L_{s-h}^* \Phi(\cdot,s-h)x](s)\|_Y
\notag\\
& \lesssim_T \|e^{hA}x-x\|_Y+\|L_{s-h}^* \Phi(\cdot,s-h)x\|_{L^2(s-h,s;U)}    
\notag\\
& \lesssim_T \|e^{hA}x-x\|_Y+\underbrace{\|\Phi(\cdot, s-h)x\|_{L^2(s-h,s;Y)}}_{S_2}\,.   
\label{e:left-continuity} 
\end{align}
holds true.
In the last passage we used that since
\begin{equation*}
L_{s-h}\in \cL(L^2(s-h,s;U),L^p(s-h,s;Y)) \qquad \forall p\,, 1\le p\le +\infty\,,
\end{equation*}
then in particular
\begin{equation*}
L_{s-h}^*\in \cL(L^1(s-h,s;Y),L^2(s-h,s;U))\,.
\end{equation*}

Thus, for the latter summand in the right hand side of \eqref{e:left-continuity} we have 
\begin{equation} \label{e:S_2}
S_2^2:=\int_{s-h}^s \|\Phi(\sigma,s-h)x\|_Y^2\,d\sigma
\lesssim_T \|x\|^2 \int_{s-h}^s 1\,d\sigma=c_T^2  \|x\|^2 h
\end{equation}
for some $c_T>0$.
Therefore, the estimate \eqref{e:S_2} combined with \eqref{e:left-continuity} allows to establish 
\begin{equation*}
\|\Phi(s,s-h)x-x\| \longrightarrow 0\,, \quad \textrm{as $h\to 0^+$}\,,
\end{equation*}
that is nothing but the left-continuity of $s \longrightarrow \Phi(t,s)$.
The proof is complete.
\end{proof}


\subsubsection{The optimal cost operator, feedback formula}
On the basis of the preliminary analysis performed in the previous sections, we are now able to
bring the former representation formula \eqref{e:pre-feedback} for the optimal control to the
definitive {\em closed-loop} form.

We compute the optimal cost by evaluating the functional \eqref{e:family-cost} on the optimal control via the obtained representation formula \eqref{e:pre-feedback} (in terms of the optimal evolution).
For any given $x\in Y$, we have 
\begin{equation*}
\begin{split}
J_t(\hat{g})&=\alpha \int_t^T \|\hat{g}(s,t,x)\|^2_U\,ds+\|\hat{y}(T,t,x)\|_Y^2  
= \frac{1}{\alpha}\int_{t}^T \|L_{tT}^*\Phi(T,t)x\|^2_U\,dt +\|\Phi(T,t)x\|_Y^2
\\
&= \frac{1}{\alpha}(\Phi(T,t)x,L_{tT}L^*_{tT}\Phi(T,t)x)_Y +\|\Phi(T,t)x\|_Y^2
\\
&=(\Phi(T,t)x,- L_{tT}\hat{g}(\cdot,t,x)+\Phi(T,t)x)_Y
\\
&=(\Phi(T,t)x,e^{A(T-t)}x)_Y=:(P(t)x,x)_Y\,,
\end{split}
\end{equation*}  
where we introduced the {\em optimal cost operator}
\begin{equation} \label{e:riccati-operator}
P(t)x = e^{A^*(T-t)}\Phi(T,t)x\,,
\end{equation} 
which is more precisely a family of continuous linear operators on $Y$ for $t\in [0,T]$. 
Then we also have, for $x,y\in Y$,
\[
\begin{split}
(P(t)x,y)_Y &= (\Phi(T,t)x,e^{A(T-t)}y)_Y=
(\Phi(T,t)x,\Phi(T,t)y - L_{tT}\hat{g}(\cdot,t,y))_Y 
\\& 
= (\Phi(T,t)x,\Phi(T,t)y)_Y + \frac{1}{\alpha} (\Phi(T,t)x, L_{tT}L^\ast_{tT}\Phi(T,t)y )_Y
\\&
= (\Phi(T,t)x,\Phi(T,t)y)_Y +\frac{1}{\alpha} \int_t^T (L_{tT}^*\Phi(T,t)x,L_{tT}^*\Phi(T,t)y)_U\,dt =(x,P(t)y)_Y\,.
\end{split}
\]
which in turn indicates that $P(t)$ is self-adjoint and non-negative.
 
At this point the evolution property \eqref{e:evolution} satisfied by the optimal state allows to disclose the presence of the optimal cost operator in the representation formula \eqref{e:pre-feedback}, so as to attain the sought feedback representation of the optimal control.
Indeed,
\begin{equation}\label{e:feedback-formula}
\begin{split}
\hat{g}(t,s,x) &= -\frac{1}{\alpha}B^* e^{(T-t)A^*}\Phi(T,s)x
-\frac{1}{\alpha}B^* e^{(T-t)A^*}\Phi(T,t)\Phi(t,s)x 
\\
& = -\frac{1}{\alpha} B^* P(t)\Phi(t,s)x =
-\frac{1}{\alpha} B^* P(t)\hat{y}(t,s,x)\,,   \quad \textrm{for $0\le s\le t\le T$,}
\end{split}
\end{equation}
and infer that the operator $B^\ast P(t)$ is continuous from $Y$ into $L^2(s,T;U)$, 
$0\le s\le t\le T$.


\medskip

{\bf Conclusions.}
Below we summarize the conclusions of the analysis carried out so far:  
\begin{enumerate}

\item[i)]
the existence of a unique optimal pair $(\hat{g}(\cdot,s,y_0),\hat{y}(\cdot,s,y_0))$ for the 
family of optimal control problems (depending on the parameter $s\in [0,T)$) follows
by pretty elementary computations, as the functional is coercive in the space of admissible controls
$\mathcal{U}=L^2(0,T;U)$;

\item[ii)]
a (pointwise in time) representation of the optimal control $\hat{g}(\cdot,s,y_0)$ in {\em closed loop form} -- namely, depending on the optimal state $\hat{y}(\cdot,s,y_0)=\Phi(t,s)y_0$ -- holds true; see \eqref{e:feedback-formula};
the gain operator $B^*P(t)$ that occurs in the said formula is well-defined on the optimal evolution;

\item[iii)]
however, in the presence of an the operator $G$ which does not possess smoothing properties (such as the identity $I$ on $Y$), as it is well-known, there is no clue that the optimal cost operator $P(t)$ solves the corresponding Riccati equation on $[0,T)$.
The possibility of giving a full meaning to the gain operator $B^*P(t)$ when acting on elements in the state space $Y$, or in a dense set of it -- and it would suffice in $\mathcal{D}(A)$ -- is an unanswered question at this point in time.
\end{enumerate}

\smallskip
This weak point of the Riccati theory pertaining to the optimal boundary control of the Maxwell system -- just like the case of boundary control systems governed by other partial differential equations of hyperbolic type -- is compensated, although only partially, proving that $P(t)$ is uniquely determined as the limit of a sequence $\{P_n(t)\}_n$ of bounded operators which do solve a differential Riccati equation, with bounded gain $B^*P_n(t)$ (for each $n\in \mathbb{Z}^+$).
This approximation procedure may lead to introduce a concept of generalized solution to Riccati equations with unbounded gains, as suggested in \cite[\S~9.5.2]{las-trig-redbooks}.
We provide the details of this approximation in the next section, for the sake of completenesss and clarity.


\section{Approximating $P(t)$ via solutions to proper Riccati equations} \label{s:approximating}
Given the optimal control problem \eqref{e:cauchy-family}-\eqref{e:family-cost}, let $P(t)$,
$t\in [0,T]$, be the relative optimal cost (or, Riccati) operator: to wit, the operator that defines the quadratic form $(P(t)y_0,y_0)_Y$ arising from the computation of the optimal value of the functional $J_t(g)$; $P(t)$ is represented in terms of the optimal evolution via \eqref{e:riccati-operator}. 
More generally, if the penalization of the state at the final time $T$ in the cost functional \eqref{e:cost} is $\|Gy(T)\|_Y^2$ 
in place of $\|y(T)\|_Y^2$, with $G\in \cL(Y)$, then the formula for the optimal cost operator is
\begin{equation} \label{e:riccati-op-general}
P(t)x = e^{A^*(T-t)}G^*G\Phi(T,t)x\,;
\end{equation} 
\eqref{e:riccati-operator} is a special case of \eqref{e:riccati-op-general}, consistent with our original choice $G=I$.

We introduce then the sequence $\{G_n\}_n=G nR(n,A)$, $n\in \mathbb{Z}^+$, where $R(\lambda,A)$ indicates the resolvent operator $(\lambda I-A)^{-1}$ meaningful for $\lambda \in \rho(A)$, and -- as it is well known -- the
bounded operators $nR(n,A)$ are {\em approximants of the identity} $I$ (the standard notation $J_n$ for these approximants is not adopted in this work, as the letter $J$ pertains to the integral functionals). 

For this, we recall that $n\in \rho(A)$ is valid here for all $n\in \mathbb{Z}^+$: indeed,
this is seen e.g. going through the proof of maximal dissipativity of the operator $A$ given in \cite[Proposition\,1]{el19} and addressing the static PDE problem corresponding to 
\[
(\lambda I-A)\begin{bmatrix}u\\v\end{bmatrix}=\begin{bmatrix}f\\g\end{bmatrix}\in Y
\] 
for any real $\lambda \ge 1$, rather than just when $\lambda =1$.

With $J_n(g)$ the functional which displays $G_n$ in place of $G$, that is 
\begin{equation}\label{e:cost_n}
J_n(g)=J_{T,n}(g) = \alpha \int_0^T  \|g(t)\|^2_U\,dt + \|G_ny(T)\|^2_Y\,,
\end{equation}
we consider the following family of optimal control problems.

\smallskip
\paragraph{\bf Problem $\mathbf{\mathcal{P}}_n$}
{\em 
Given $y_0\in Y$, seek a control function $\hat{g}_n(\cdot)$ 
that minimizes the functional $J_{T,n}(g)$ overall $g\in L^2(0,T;U)$, 
where $y(\cdot)=y(\cdot;y_0,g)$ is the solution to \eqref{e:cauchy-family}
corresponding to the control function $g(\cdot)$ and with initial datum $y_0$ given
by \eqref{e:mild-family}.
}

\smallskip
Owing to the enhanced regularity of the operator $G_n$, the optimal control problem $\mathcal{P}_n$
possesses a full, Riccati-based, solution; see \cite[vol.\,II, Theorems~9.2.1-9.2.2]{las-trig-redbooks}.
In particular, there is a family of linear and bounded, non-negative, self-adjoint operators 
$\{P_n(t)\}_n$, $n\in \mathbb{Z}^+$, whose terms solve uniquely the differential Riccati equation
\begin{equation} \label{e:DRE_n}
\frac{d}{dt}\langle P_n(t)x,y\rangle_Y + \langle P_n(t)x,Ay\rangle_Y + \langle Ax,P_n(t)y\rangle_Y   
- \langle B^*P_n(t)x,B^*P_n(t)y\rangle_U=0 \quad \text{in $[0,T)$}
\end{equation}
for $x,y \in \cD(A)$, supplemented with the final condition $P_n(T)=G_n$, with gain operator 
$B^*P_n(t)$ that belongs to $\cL(Y,U)$ (for all $t\in [0,T]$). 

\begin{remark}
\begin{rm}
Given $s\in [0,T)$, the notation $J_{s,n}=J_{T,s,n}(g)$ for the parametrized quadratic functional 
is self-explanatory.
\end{rm}
\end{remark}

The outcome of the limit process, as $n\to +\infty$, in relation to the original optimal control problem is detailed in the following result, which is akin to Theorem~9.5.2.2 in 
\cite[vol.\,II]{las-trig-redbooks}.

\begin{proposition} \label{e:extended-sln}
Given the original and the approximating optimal control problems $\mathcal{P}$ and $\mathcal{P}_n$, let $J_s(g)$, $\hat{g}(\cdot,s,y_0)$, $\hat{y}(\cdot,s,y_0)$, $P(s)$ 
and $J_n(g)$, $\hat{g}_n(\cdot,s,y_0)$, $\hat{y}_n(\cdot,s,y_0)$, $P_n(s)$ be the functional to be minimized, the optimal control and state, the optimal cost operator pertaining to either problem, respectively. 

Then, the following convergence results hold true, as $n\to +\infty$:
\begin{subequations} \label{e:convergence}
\begin{align}
&\text{$\|\hat{g}_n(\cdot,s,y_0)-\hat{g}(\cdot,s,y_0)\|\longrightarrow 0$ 
\; in $L^2(s,T;U)$, uniformly in $s$;}
\label{e:convergence_1}
\\[2mm]
&\text{$\|\hat{y}_n(\cdot,s,y_0)-\hat{y}(\cdot,s,y_0)\|\longrightarrow 0$ 
\; in $C([s,T],U)$, uniformly in $s$;}
\label{e:convergence_2}
\\[2mm]
& \text{$J_{s,n}(\hat{g}_n)-J_{s}(\hat{g})\longrightarrow 0$ 
\; in $\mathbb{R}$, uniformly in $s$};
\label{e:convergence_3}
\\[2mm]
&\text{$\|P_n(\cdot)z-P(\cdot)z\|\longrightarrow 0$ 
\; in $C([0,T],Y)$, for all $z\in Y$.}
\label{e:convergence_4}
\end{align}
\end{subequations}

\end{proposition}

\begin{proof}
{\bf a.}
Let us begin by recalling the expression of the optimal solution for problems $\mathcal{P}$ and 
$\mathcal{P}_n$, respectively, in terms of the initial state $y_0$.
We have
\begin{equation*}
\hat{g}(\cdot,s,y_0)= -\big[\Lambda_{sT}^{-1}L_{sT}^* G^*G e^{(T-s)A}y_0\big](\cdot)\,,
\quad 
\hat{g}_n(\cdot,s,y_0)= -\big[\Lambda_{sT,n}^{-1}L_{sT}^* G_n^*G_n e^{(T-s)A}y_0\big](\cdot)\,,
\end{equation*}
where
\begin{equation*}
\Lambda_{sT} =I+L_{sT}^*G^*GL_{sT}\,, \qquad
\Lambda_{sT,n} =I+L_{sT}^*G_n^*G_nL_{sT}\,,
\end{equation*}
so that 
\begin{equation*}
\Lambda_{sT,n}^{-1}- \Lambda_{sT}^{-1} 
= \Lambda_{sT,n}^{-1}(\Lambda_{sT}-\Lambda_{sT,n})\Lambda_{sT}^{-1}
= \Lambda_{sT,n}^{-1}L_{sT}^*(G^*G- G_n^*G_n)L_{sT}\Lambda_{sT}^{-1}\,.
\end{equation*}
We see that
\begin{equation} \label{e:diff-inverse}
\|\Lambda_{sT,n}^{-1}- \Lambda_{sT}^{-1}\|_{\mathcal{L}(Y)} \longrightarrow 0\,, 
\quad \text{as $n\to+\infty$,}  
\end{equation}
as a consequence of the strong convergence of $G_n^*G_n$ to $G^*G$ and since 
$\|\Lambda_{sT,n}^{-1}\|_{\mathcal{L}(Y)}\le 1/\alpha$ for all $n\in \mathbb{Z}^+$.

We compute the difference $\hat{g}_n(\cdot,s,y_0)-\hat{g}(\cdot,s,y_0)$, that we decompose
as follows:
\begin{equation*}
\begin{split}
&\hat{g}_n(\cdot,s,y_0)-\hat{g}(\cdot,s,y_0) \\[1mm]
&=-\big[\Lambda_{sT,n}^{-1}L_{sT}^* G_n^*G_n e^{(T-s)A}y_0\big](\cdot) +\big[\Lambda_{sT}^{-1}L_{sT}^* G^*G e^{(T-s)A}y_0\big](\cdot)
\\[1mm]
& =-\Big[\big(\Lambda_{sT,n}^{-1}- \Lambda_{sT}^{-1} \big) L_{sT}^*G_n^*G_n
+ \Lambda_{sT}^{-1} L_{sT}^*(G_n^*G_n-G^*G)\Big]\,e^{(T-s)A}y_0\,.
\end{split}
\end{equation*}
Thus, the convergence in \eqref{e:convergence_1} follows taking into account 
\eqref{e:diff-inverse} and the regularity $L_{sT}^*\in \mathcal{L}(Y,L^2(s,T;U))$,
besides -- once again -- the strong convergence of $G_n^*G_n$ to $G^*G$ and the uniform bound
$\|\Lambda_{sT,n}^{-1}\|_{\mathcal{L}(Y)}\le 1/\alpha$.

\smallskip
\noindent
{\bf b.}
The convergence of the sequence of optimal states follows from the one of the corresponding
optimal controls readily.
Indeed, it suffices to write down (for $s,t$ such that $0\le s\le t\le T$)
\begin{equation*}
\begin{split}
\hat{y}_n(t,s,y_0)-\hat{y}(t,s,y_0)
&=e^{(t-s)A}y_0+ [L_s\hat{g}_n(\cdot,s,y_0)](t) - \big(e^{(t-s)A}y_0+ [L_s\hat{g}(\cdot,s,y_0)](t)\big)
\\[1mm]
& =\big[L_s\big(\hat{g}_n(\cdot,s,y_0) -\hat{g}(\cdot,s,y_0)\big)\big](t)
\end{split}
\end{equation*}
and recall that $L_s\in \mathcal{L}(L^2(s,T;U),C([s,T],Y))$ (with convergence which is uniform with respect to $s$), thus confirming the validity of \eqref{e:convergence_2}.

\smallskip
\noindent
{\bf c.}
Given \eqref{e:convergence_1} and \eqref{e:convergence_2}, the limit \eqref{e:convergence_3} -- that is the convergence of the sequence of optimal values $J_{s,n}(\hat{g}_n)$ to $J_s(\hat{g})$, as $n\to +\infty$ -- is straightforward, as
\begin{equation*}
\min_{g\in L^2(s,T;U)} J_{s,n}(g)=J_{s,n}(\hat{g}_n)
= \alpha \int_s^T  \|\hat{g}_n(t)\|^2_U\,dt + \|G_n\hat{y}_n(T)\|^2_Y\,.
\end{equation*}
The convergence of $J_n(\hat{g})$ to $J(\hat{g})$
holds as a special case ($s=0$).

\smallskip
\noindent
{\bf d.}
Finally, resuming the definition of any optimal cost operator in terms of the respective optimal evolution, we see that
\begin{equation*}
\begin{split}
P_n(t)y_0-P(t)y_0
& = e^{(T-t)A^*}G_n^*G_n\hat{y}_n(T,t,y_0) - e^{(T-t)A^*}G^*G\hat{y}(T,t,y_0)
\\[1mm]
& = e^{(T-t)A^*}\big[G_n^*G_n-G^*G\big]\hat{y}_n(T,t,y_0) 
\\[1mm]
& \qquad + e^{(T-t)A^*} G^*G \big[\hat{y}_n(T,t,y_0) -\hat{y}(T,t,y_0)\big]\,;
\end{split}
\end{equation*}
then, \eqref{e:convergence_4} follows considering 
\begin{itemize}
\item
(for the first summand) the bound for $\|\hat{y}_n(T,t,y_0)\|=\|\Phi_n(T,t)y_0\|\le C_T \|y_0\|$, uniform with respect to $n$, along with the strong convergence of $G_n^*G_n$ to $G^*G$, combined with
\item
(for the second summand) the convergence \eqref{e:convergence_2}.
\end{itemize}
\end{proof}


\section{The optimal solution in the case of zero conductivity} \label{s:zero-conductivity}
In this section we examine the special case where the conductivity $\sigma$ in the Maxwell system \eqref{e:maxwell} is zero.
It is readily seen from \eqref{e:A} and  \eqref{e:A-star} that in this case the dynamics operator
$A$ is skew-adjoint, i.e. $A^*=-A$, so that $A$ is the infinitesimal generator of a $C_0$-group 
$\{e^{tA}\}_{t\in\mathbb{R}}$ by Stone's Theorem; see e.g. \cite{engel-nagel}.
Then, as we will see below, the work \cite{flandoli_1987} allows us not only to find $P(t)$ by way of solving the dual RE in \eqref{e:Cauchy-for-DRE}, thus obtaining the synthesis of the optimal control via the {\em closed-loop} equation, but also to infer an {\em open-loop} representation of the optimal solution (which makes proving the well-posedness of the Riccati equation \eqref{e:Cauchy-for-DRE} dispensable).

A property that is ascertained and that plays a key role in the outcome for the optimal solution 
of problem \eqref{e:full-cauchy}-\eqref{e:abstract-general-cost} is the fact that $P(t)$ is an isomorphism for any $t\in [0,T]$.
This property is deduced in \cite{flandoli_1987} by establishing
\begin{itemize}
\item
the well-posedness of a certain Cauchy problem associated with the dual Riccati equation, that
is \eqref{e:Cauchy-for-dualRE}; 
\item
that the unique solution $Q(t)$ of the aforesaid Cauchy problem is an isomorphism for all 
$t\in [0,T]$;
\item
that the inverse operator $Q(t)^{-1}$, $t\in [0,T]$, coincides with the (Riccati) optimal cost operator $P(t)$ for the original optimal control problem.
\end{itemize}
Actually, that $P(t)$ is an isomorphism for all $t\in [0,T]$ can be inferred {\em a priori}.
We include this result at the outset, pointing out that although the proof is similar to the proof of Theorem 9.4.1 in \cite{las-trig-redbooks}, 
our setting is different since we work with $G=I$ and $C=0$.
A certain sought estimate from below can be attained just using that $G$ is invertible. 


\begin{proposition}[$P(t)$ is an isomorphism] \label{p:isomorphism}
Let $P(t)$ be the optimal cost (or, Riccati) operator of the optimal control problem over the interval
$[t,T]$, namely, the operator that defines the quadratic form $(P(t)y_0,y_0)_Y$ arising from the computation of the optimal value of the functional $J_t(\hat{g})$, see formula \eqref{e:optimal-cost_1}.
%
Then, $P(t)$ is an isomorphism for $t\in [0,T]$.
\end{proposition} 

\begin{proof}
We know that the optimal cost of the optimal control problem with $t$ as initial time and $y_0$ as initial state is
\begin{equation} \label{e:1}
\langle P(t)y_0,y_0\rangle =J_t(\hat{g}) 
=\alpha \|\hat{g}(\cdot,t,y_0)\|_{L^2(t,T;U)}^2+\|\hat{y}(T,t,y_0)\|_Y^2
\ge \|\hat{y}(T,t,y_0)\|_Y^2\,.
\end{equation}
(We are in the case in which $C=0$ and $G=I$: however, the proof would work -- {\em mutatis mutandis} -- with nontrivial $C$ as well, provided $G$ is an isomorphism.)
Recall now the relation \eqref{e:pre-feedback} linking the optimal control to the optimal state,
which yields \eqref{e:to-be-used-in-5}, that is equivalent to  
\begin{equation*}
\Big[I+\frac{1}{\alpha}L_{tT}L_{tT}^*\Big] \hat{y}(T,t,y_0)= e^{(T-t)A}y_0\,.
\end{equation*}
Thus, inserting 
\begin{equation*}
\hat{y}(T,t,y_0)= \Big[I+\frac{1}{\alpha}L_{tT}L_{tT}^*\Big]^{-1} e^{(T-t)A}y_0
\end{equation*}
in \eqref{e:1} we see that 
\begin{equation*}
\langle P(t)y_0,y_0\rangle \ge \Big\|\Big[I+\frac{1}{\alpha}L_{tT}L_{tT}^*\Big]^{-1} 
e^{(T-t)A}y_0\Big\|_Y^2\,.
\end{equation*}
The above implies
\begin{equation} \label{e:2}
\langle P(t)y_0,y_0\rangle \gtrsim_T \,\|e^{(T-t)A}y_0\|_Y^2 \ge \|y_0\|_Y^2
\qquad \forall t\in [0,T]\,,
\end{equation}
where the last estimate is a consequence of $A^*=-A$, along with
\begin{equation*}
e^{(T-t)A^*}e^{(T-t)A}= e^{-(T-t)A}e^{(T-t)A}=I\,.
\end{equation*}
The conclusion follows combining \eqref{e:2} with
\begin{equation*}
\langle P(t)y_0,y_0\rangle \lesssim_T \|y_0\|_Y^2 \qquad \forall t\in [0,T]\,.
\end{equation*} 

\end{proof}
The following result adds a piece to the picture of solvability of the optimal control problem
\eqref{e:maxwell}-\eqref{e:IC}-\eqref{e:BC}-\eqref{e:cost} for the Maxwell's system. 


\begin{proposition} \label{p:from-Flandoli}
Consider the Maxwell's system \eqref{e:maxwell} with $\sigma=0$, the control system \eqref{e:full-cauchy} which is the abstract reformulation of the IBVP \eqref{e:maxwell}-\eqref{e:IC}-\eqref{e:BC} in $Y$, along with the quadratic funtional \eqref{e:abstract-general-cost}.
Assume $C, G\in \cL(Y)$ and that in addition $G$ is an isomorphism.

Then, the following assertions are valid for the optimal control problem \eqref{e:full-cauchy}-\eqref{e:abstract-general-cost}.

\smallskip
a) The Riccati operator $P(t)$ that occurs in the feedback formula \eqref{e:feedback} coincides
with $Q(t)^{-1}$, $Q(t)$ being the unique solution to the Cauchy problem associated with the dual Riccati equation \eqref{e:Cauchy-for-dualRE}.

\smallskip
b) The optimal pair $(\hat{g},\hat{y})$ admits the following (more explicit, alternative) representation:
\begin{equation}
\hat{g}(t)\equiv \hat{g}(t,0,y_0)=-B^*z(t)\,, \quad \hat{y}(t)=Q(t)z(t)\,,
\end{equation}
where $Q(t)$ is as above, while $z(\cdot)$ solves the non-autonomus system 
\begin{equation}
\begin{cases}
z'=[A-C^*CQ(t)]z
\\
z(0)=Q(0)^{-1}y_0\,.
\end{cases}
\end{equation}
\smallskip
c) Therefore, in particular, if $C=0$ and $G=I$ -- that is precisely the case of the original functional (with $\alpha =1$) -- we simply have
\begin{equation}
\hat{g}(t)=-B^*e^{tA}Q(0)^{-1}y_0\,, \quad \hat{y}(t)=Q(t)e^{tA}Q(0)^{-1}y_0\,,
\end{equation}
where
\begin{equation} \label{e:explicit-Q}
Q(t)x=\int_t^T e^{(r-t)A^*}BB^*e^{(r-t)A}x\,dr+e^{(T-t)A^*}e^{(T-t)A}x\,, 
\quad x\in Y\,, \; t\in [0,T]\,.
\end{equation}

\end{proposition}

\begin{proof}
All the assertions are (pretty direct) consequences of the results contained in \cite{flandoli_1987}.
The well-posedness of the Cauchy problem \eqref{e:Cauchy-for-dualRE} for the dual RE is shown
in \cite[Theorem~2.1.1]{flandoli_1987}.
The assertion in a) is contained in the proof of Theorem~2.3.1 in \cite{flandoli_1987}.
The assertion in b) is essentially part of the discussion in \cite[Section~2.4]{flandoli_1987}.

The conclusions for the special case $C=0$ and $G=I$ -- our main interest and hence stated separately in c) -- follow starting from the outcomes in b), taking into account that in the case $C=0$ 
\begin{itemize}
\item
the evolution operator $V(t,s)$ in \cite[Theorem~2.1.1]{flandoli_1987} reduces to 
$e^{-(t-s)A^*}\equiv e^{(t-s)A}$;
\item
the (operator) quadratic integral equation $(2.5)$ in \cite[Theorem~2.3.1]{flandoli_1987} in the unknown $Q(t)$ becomes an actual explicit formula for $Q(t)$, that is \eqref{e:explicit-Q}.
\end{itemize}
This concludes the proof of the proposition.
\end{proof}

\medskip


\section*{Acknowledgements}
A major part of this research was carried out while Francesca Bucci was the guest of Matthias Eller at Georgetown University in Washington, DC. Bucci's stays were supported by Georgetown University, which both authors gratefully acknowledge. 

F.\,Bucci wishes to thank the Department of Mathematics and Statistics of the Georgetown University, for its hospitality.
The research of F.\,Bucci has been performed in the framework of the MIUR-PRIN Grant 2020F3NCPX ``Mathematics for Industry 4.0 (Math4I4)'', whose support is acknowledged.
This work was supported by the Gruppo Nazionale per l'Analisi Mate\-ma\-tica, la Probabilit\`a  e le loro Applicazioni of the Istituto Nazionale di Alta Matematica (GNAMPA-INdAM).


\end{document}